\documentclass[12pt]{amsart}
\usepackage[margin=1in,includeheadfoot]{geometry}
\usepackage{amsmath}
\usepackage{slashed}
\usepackage{graphicx}

\usepackage{mathrsfs}
\usepackage{txfonts}
\usepackage{amssymb,dsfont}
\usepackage{enumerate}
\usepackage{slashed}
\usepackage{fancyhdr}
\usepackage{aliascnt}
\usepackage{pdfsync}

\newtheorem{thm}{Theorem}[section]
\newtheorem{cor}[thm]{Corollary}
\newtheorem{lem}[thm]{Lemma}

\theoremstyle{definition}
\newtheorem{defn}[thm]{Definition}
\theoremstyle{remark}

\numberwithin{equation}{section}

\newcounter{stepnum}

\def\bee{\begin{eqnarray}}
\def\beee{\begin{eqnarray*}}
\def\eee{\end{eqnarray}}
\def\eeee{\end{eqnarray*}}

\def\ba{\begin{array}}
\def\ea{\end{array}}
\def\R{\mathbb R}

\begin{document}

\title[Lorentzian harmonic maps]{Partial regularity of harmonic maps from a Riemannian manifold into a Lorentzian manifold}%

\author[Li]{Jiayu  Li}
\address{School of Mathematics Sciences, University of Science and Technology of China\\ Hefei 230026,
Anhui, China}
\email{jiayuli@ustc.edu.cn}

\author[Liu]{Lei Liu}%
\address{Max Planck Institute for Mathematics in the Sciences\\ Inselstrasse 22\\ 04103 Leipzig, Germany}
\email{leiliu@mis.mpg.de or llei1988@mail.ustc.edu.cn} %

\thanks{The research is supported by  NSF in China No 11426236, 11131007.}

\subjclass[2010]{53C43, 58E20}
\keywords{Lorentzian harmonic map, Stationary, Partial regularity, Blow-up.}

\date{\today}
\begin{abstract}
In this paper, we will study the partial regularity theorem for stationary harmonic maps from a Riemannian manifold into a Lorentzian manifold. For a weakly stationary harmonic map $(u,v)$ from a smooth bounded open domain $\Omega\subset\R^m$ to a Lorentzian manifold with Dirichlet boundary condition, we prove that it is smooth outside a closed set whose $(m-2)$-dimension Hausdorff measure is zero. Moreover, if the target manifold $N$ does not admit any harmonic sphere $S^l$, $l=2,...,m-1$, we will show $(u,v)$ is smooth.
\end{abstract}
\maketitle
\section{introduction}

Suppose $(M,g)$ and $(N,h_N)$ are two compact Riemannian manifolds of dimension $m$ and $n$ respectively. For a map $u\in C^1(M, N)$, the energy functional of $u$ is defined as
\begin{equation}
E(u)=\frac{1}{2}\int_M |\nabla u|^2dvol_g.
\end{equation}
A critical point of the energy functional $E$ is called a harmonic map. By Nash's embedding theorem, we can embed $N$ isometrically into some Euclidian space $\R^K$ and the corresponding Euler-Lagrange equation is
$$
\Delta_g u=A(u)(\nabla u,\nabla u),
$$
where $\Delta_g$ is the Laplace-Beltrami operator on $M$ with respect to $g$ and $A$ is the second fundamental form of $N\subset \mathbb{R}^K$.

Harmonic map is a very important notion in geometric analysis which has been widely studied in the past decades. Physically, harmonic map comes from the nonlinear sigma model, which plays an important role in quantum field and string theory. From the perspective of general relativity, it is nature to consider that the target of harmonic maps are Lorentzian manifolds. Geometrically, the link between harmonic maps into $S^4_1$ and the conformal Gauss maps of Willmore surface in $S^3$ also leads to such harmonic maps \cite{bryant}. The work on minimal surfaces in anti-de-Sitter spaces and its applications in theoretical physics also shows the importance of such maps \cite{aldaymaldacena}. In this paper, we shall focus on the interior partial regularity of stationary harmonic maps from a compact Riemannian manifold of dimension $m\ (\geq 3)$ into a Lorentzian manifold.

We now proceed to introduce the model. Let $N\times\R$ be a Lorentzian manifold equipped with a warped product metric $$h=h_N-\beta(d\theta)^2,$$ where $(\R,d\theta^2)$ is the standard $1$-dimensional Euclidean space and $\beta$ is a positive smooth function on $(N,h_N)$. Since $N$ is compact, there exist positive constants $\lambda_1$ and $\lambda_2$ such that$$0<\lambda_1\leq\beta(y)\leq\lambda_2<\infty\ \ and\ \ |\nabla\beta(y)|\leq\lambda_2,\ \forall\ y\in N.$$

Denote
\[
W^{1,2}(M,N\times\R):=\{u\in W^{1,2}(M,\R^K),\ v\in W^{1,2}(M,\R)|u(x)\in N\ for\ a.e.\ x\in M\}.
\]
For  $(u,v)\in W^{1,2}(M,N\times {\mathbb{R}})$, we consider the following functional
\begin{equation}\label{lag}
E_h(u, v; M)= \frac{1}{2}\int_{M} \left\{|\nabla u|^2- \beta(u)|\nabla v|^2 \right\} dvol_g,
\end{equation}
which is called the Lorentzian energy of the map $(u,v)$ on $M$. A critical point $(u,v)$ of the functional \eqref{lag} is called a harmonic map from $(M,g)$ into the Lorentzian manifold  $(N\times{\mathbb{R}},h)$.

When the target manifold is a Lorentzian manifold, the existence of geodesics was studied in \cite{bencifortunatogiannoni} and Greco constructed a smooth harmonic map via some developed variational methods in \cite{greco1,greco2}. Recently, Han-Jost-Liu-Zhao \cite{HJLZ} investigated a parabolic-elliptic system for maps and got a global existence result by assuming either some geometric conditions on the target manifold or small energy of the initial maps. The result implies the existence of a harmonic map in a given homotopy class. The blowup behavior for Lorentzian harmonic maps was studied in \cite{hanzhaozhu} and for approximate Lorentzian harmonic maps or Lorentzian harmonic maps flow from a Riemann surface were studied in \cite{HJLZ,HJLZ-2}. The regularity theory was studied in \cite{Isobe-2,zhu} for dimension two and in \cite{Isobe-1} for higher dimension on some kinds of minimal type solutions.

Via direct calculations, one can derive the  Euler-Lagrange equations for \eqref{lag},

\begin{align}\label{equation:LHM}
\begin{cases}
-\Delta u= A(u)(\nabla u,\nabla u)-B^\top(u)|\nabla v|^2,\ &in\ M \\
-div (\beta(u)\nabla v)=0, \ &in\ M
\end{cases}
\end{align}
where $A$ is the second fundamental form of $N$ in ${\mathbb{R}}^K$, $B(u):=(B^1, B^2, \cdots, B^K)$ with
$$
B^j:=-\frac{1}{2}\frac{\partial\beta(u)}{\partial y^j}
$$
and $B^{\top}$ is the tangential part of $B$ along the map $u$. For  details, see \cite{zhu,Isobe-2}.

\

\begin{defn}
We call $(u,v)\in W^{1,2}(\Omega,N\times \R)$ a weakly Lorentzian harmonic map with Dirichlet boundary data $$(u,v)|_{\partial \Omega}=(\phi,\psi),$$ if it is a weak solution of equation \eqref{equation:LHM} with boundary data $(\phi,\psi)$.
\end{defn}

\


Similar to harmonic maps, we introduce the notion of stationary Lorentzian harmonic maps.
\begin{defn}\label{Def:stationary}
A weakly Lorentzian harmonic map $(u,v)\in W^{1,2}(\Omega,N\times\R)$ is called a stationary Lorentzian harmonic map, if it is also a critical point of $E_h$ with respect to the domain variations, $i.e.$ for any $Y\in C^\infty_0(\Omega,\R^m)$, it holds $$\frac{d}{dt}|_{t=0}\int_{\Omega}\frac{1}{2}\left(|\nabla u_t|^2-\beta(u_t)|\nabla v_t|^2 \right)dvol_g=0,$$ where $u_t(x)=u(x+tY(x))$ and $v_t(x)=v(x+tY(x))$.
\end{defn}


\

Our first main result is the following small energy regularity theorem.
\begin{thm}\label{thm:small-energy-regularity}
For $m\geq 2$ and any $\alpha\in (0,1)$, there exists an $\epsilon_0>0$ depending only on $m$, $\alpha$ and $(N,h_N)$ such that if $(u,v)\in W^{1,2}(\Omega,N\times\R)$ is a weakly Lorentzian harmonic map satisfying
\begin{align}
\sup_{x\in B_{r_0}(x_0),0<r\leq r_0}r^{2-m}\int_{B_r(x)}|\nabla u|^2dvol_g\leq\epsilon^2_0,
\end{align}
then $(u,v)\in C^\infty(B_{\frac{r_0}{2}}(x_0))$. Moreover, it satisfies the following estimate that
\begin{align}\label{inequality:09}
&r_0\|\nabla u\|_{L^\infty(B_{r_0/2}(x_0))}+r_0\|\nabla v\|_{L^\infty(B_{r_0/2}(x_0))} +r_0^{1+\alpha}\|\nabla u\|_{C^\alpha(B_{r_0/2}(x_0))}+r_0^{1+\alpha}\|\nabla v\|_{C^\alpha(B_{r_0/2}(x_0))}\notag\\
&\leq C\left(r_0^{1-\frac{m}{2}}\|(\nabla u,\nabla v)\|_{L^2(B_{r_0}(x_0))}+r_0^{2-m}\|(\nabla u,\nabla v)\|^2_{L^2(B_{r_0}(x_0))}+r_0^{4-2m}\|\nabla v\|^4_{L^2(B_{r_0}(x_0))}\right),\end{align} where $C=C(m,\lambda_1,\lambda_2,\alpha, N)$ is a positive constant and $$\|(\nabla u,\nabla v)\|_{L^2(B_{r_0}(x_0))}^2:=\|\nabla u\|^2_{L^2(B_{r_0}(x_0))}+\|\nabla v\|^2_{L^2(B_{r_0}(x_0))}.$$
\end{thm}

\

In this paper, we can get the following interior partial regularity theorem. For a similar result of harmonic maps, one can refer to \cite{Bethuel,Evans,Li-Tian}. For results on Gauge theory, one can refer to \cite{Tian}.
\begin{thm}\label{thm:main-1}
For $m\geq 2$, let $(u,v)\in W^{1,2}(\Omega,N\times \R)$ be a stationary Lorentzian harmonic map with Dirichlet boundary data $(u,v)|_{\partial\Omega}=(\phi,\psi)\in C^0(\partial\Omega)$, there exists a closed subset $S(u)\subset \Omega$, with $H^{m-2}(S(u))=0$, such that $(u,v)\in C^\infty(\Omega\setminus S(u))$.
\end{thm}

\

Furthermore, we have
\begin{thm}\label{thm:main-2}
Under the same assumption as above theorem, if $N$ does not admit harmonic spheres, $S^l$, $l=2,...,m-1$, then $(u,v)$ is smooth.
\end{thm}

\

To prove the partial regularity results, we first need to establish the monotonicity formula for stationary Lorentzian harmonic maps. Thanks to the elliptic estimates of $v$-equation of divergence form, we can control the additional terms (corresponds to harmonic maps) in the monotonicity formula. Secondly, we need to study the energy concentration set of a blow-up sequence of stationary Lorentzian harmonic maps. Here, we follow Lin's scheme \cite{Lin} to get the first bubble which is a nonconstant harmonic sphere. The proof is based on the analysis of defect measure using geometric measure theory.

\

The rest of paper is organized as follows. In section \ref{sec:Monotonicity formula}, we establish the monotonicity formula for stationary Lorentzian harmonic maps which is crucial in the proof of our main theorems. In section \ref{sec:proof of mainthm}, we prove the small energy regularity Theorem \ref{thm:small-energy-regularity} and then the partial regularity Theorem \ref{thm:main-1} follows immediately from a standard monotonicity formula argument. Theorem \ref{thm:main-2} will be proved in section \ref{sec:gradient estimate}.

\

\section{Monotonicity formula}\label{sec:Monotonicity formula}
In this section, we firstly derive the monotonicity formula for stationary Lorentzian harmonic maps. Secondly, for reader's convenience, we recall a regularity theorem in \cite{B-Sharp} which will be used in the proof.

\

Thanks to the divergence structure of $v$-equation, we have the following estimate.
\begin{lem}\label{lem:-02}
Let $(u,v)\in W^{1,2}(\Omega,N\times\R)$ be a weakly Lorentzian harmonic map with Dirichlet boundary data $(\phi,\psi)\in C^0(\partial\Omega)$. Then $v\in W^{1,p}(\Omega)$ for any $1<p<\infty$ and
\begin{equation}
\|\nabla v\|_{L^p}\leq C(p,\lambda_1,\lambda_2,\Omega)\|\psi\|_{C^0(\partial\Omega)}.
\end{equation}
\end{lem}
\begin{proof}
Let $v$ be the unique smooth solution of the equation
\begin{align*}
\begin{cases}
  \Delta v=0, & \mbox{ in } \Omega,\\
  v(x)=\psi,  & \mbox{ on } \partial \Omega,
\end{cases}
\end{align*}
which satisfies $$\|v\|_{C^2(\overline{\Omega})}\leq C(\Omega)\|\psi\|_{C^{0}(\partial\Omega)}.$$

We call $v$ an extension of $\psi$ and for simplicity, we still denote it by $\psi\in C^2(\overline{\Omega})$. It is easy to see that $v-\psi\in W^{1,2}_0(\Omega)$ is a weak solution of $$-div\left(\beta(u)\nabla(v-\psi)\right)=div(\beta(u)\nabla\psi).$$ By the standard theory of second elliptic operator of divergence form (cf. Theorem 1 in \cite{Meyers}), we obtain that $v\in W^{1,p}$ for any $1<p<\infty$ and satisfies $$\|\nabla v\|_{L^p}\leq C(p,\lambda_1,\lambda_2,\Omega)\|\nabla\psi\|_{L^p}\leq C(p,\lambda_1,\lambda_2,\Omega)\|\psi\|_{C^0(\partial\Omega)}.$$
\end{proof}

\

Next, we derive the stationary identity for stationary Lorentzian harmonic maps.
\begin{lem}\label{lem:-01}
Let $(u,v)\in W^{1,2}(\Omega,N\times\R)$ be a weakly Lorentzian harmonic map. Then $(u,v)$ is stationary if and only if for any $Y\in C^\infty_0(\Omega,\R^m)$, there holds
\begin{equation}\label{equation:03}
\int_\Omega\left(\langle\frac{\partial u}{\partial x^\alpha},\frac{\partial u}{\partial x^\gamma}\rangle-\beta(u)\langle\frac{\partial v}{\partial x^\alpha},\frac{\partial v}{\partial x^\gamma}\rangle-\frac{1}{2}(|\nabla u|^2-\beta(u)|\nabla v|^2)\delta_{\alpha\gamma}\right)\frac{\partial Y^\gamma}{\partial x^\alpha}dx=0.
\end{equation}
\end{lem}
\begin{proof}
For any $Y\in C^\infty_0(\Omega,\R^m)$, let $t\in\R$ small enough and $y=F_t(x):=x+tY(x)$ and $x=F_t^{-1}(y)$. By Definition \ref{Def:stationary}, $(u,v)$ is stationary if and only if $$\frac{d}{dt}|_{t=0}\int_{\Omega}\frac{1}{2}\left(|\nabla u_t|^2-\beta(u_t)|\nabla v_t|^2 \right)dx=0,$$ where $u_t(x)=u(F_t(x))$ and $v_t(x)=v(F_t(x))$.

On the one hand, by a standard calculation (see, e.g. \cite{Lin-Wang}),  we have
\begin{align}\label{equation:01}
\frac{d}{dt}|_{t=0}\frac{1}{2}\int_{\Omega}|\nabla u_t|^2dx
=
\int_\Omega\left(\langle\frac{\partial u}{\partial x^\alpha},\frac{\partial u}{\partial x^\gamma}\rangle-\frac{1}{2}|\nabla u|^2\delta_{\alpha\gamma}\right)\frac{\partial Y^\gamma}{\partial x^\alpha}dx.
\end{align}

On the other hand, computing directly, we obtain
\begin{align*}
\frac{d}{dt}|_{t=0}(\frac{1}{2}\beta(u_t)|\nabla v_t|^2)&=\frac{1}{2}\frac{\partial\beta(u)}{\partial x^\alpha}Y^\alpha |\nabla v|^2+\beta(u)\langle\frac{\partial v}{\partial x^\alpha},\frac{\partial v}{\partial x^\gamma}\rangle\frac{\partial Y^\gamma}{\partial x^\alpha}+\beta(u) \langle\frac{\partial^2 v}{\partial x^\alpha\partial x^\gamma},\frac{\partial v}{\partial x^\gamma}\rangle Y^\alpha\\
&=\frac{1}{2}\frac{\partial}{\partial x^\alpha}(\beta(u)|\nabla v|^2) Y^\alpha +\beta(u)\langle\frac{\partial v}{\partial x^\alpha},\frac{\partial v}{\partial x^\gamma}\rangle\frac{\partial Y^\gamma}{\partial x^\alpha}.
\end{align*}
Thus,
\begin{align}\label{equation:02}
\frac{d}{dt}|_{t=0}\frac{1}{2}\int_{\Omega}\beta(u_t)|\nabla v_t|^2dx
=
\int_\Omega\beta(u)\left(\langle\frac{\partial v}{\partial x^\alpha},\frac{\partial v}{\partial x^\gamma}\rangle-\frac{1}{2}|\nabla v|^2\delta_{\alpha\gamma}\right)\frac{\partial Y^\gamma}{\partial x^\alpha}dx.
\end{align}
Combing \eqref{equation:01} with \eqref{equation:02}, we will get the conclusion of the lemma.
\end{proof}

\

Now, we can derive the monotonicity formula for stationary Lorentzian harmonic maps.
\begin{lem}\label{lem:monotonicity}
Let $(u,v)\in W^{1,2}(\Omega,N\times \R)$ be a stationary Lorentzian harmonic map. Then for any $x_0\in\Omega$ and $0<r_1\leq r_2<dist(x_0,\partial\Omega)$, there holds
\begin{align*}
&r_2^{2-m}\int_{B_{r_2}(x_0)}(|\nabla u|^2-\beta(u)|\nabla v|^2)dx-r_1^{2-m}\int_{B_{r_1}(x_0)}(|\nabla u|^2-\beta(u)|\nabla v|^2)dx\notag\\
&=
2\int_{B_{r_2}(x_0)\setminus B_{r_1}(x_0)}|x-x_0|^{2-m}(|\frac{\partial u}{\partial r}|^2-\beta(u)|\frac{\partial v}{\partial r}|^2)dx
\end{align*}
where $\partial_r=\frac{\partial}{\partial r}=\frac{\partial}{\partial |x-x_0|}$.
\end{lem}
\begin{proof}
For simplicity, we assume $x_0=0\in\Omega$. For any $\epsilon>0$ and $0<r<dist(0,\partial\Omega)$, let $\varphi_\epsilon(x)=\varphi_\epsilon(|x|)\in C_0^\infty(B_r)$ be such that $$0\leq\varphi_\epsilon(x)\leq 1\ \ and \ \ \varphi_\epsilon(x)|_{B_{(1-\epsilon)r}}=1.$$ Taking $Y(x)=x\varphi_\epsilon(x)$ into the formula \eqref{equation:03} and noting that $$\frac{\partial Y^\gamma}{\partial x^\alpha}=\varphi_\epsilon(x)\delta_{\alpha,\gamma}+\frac{x^\alpha x^\gamma}{|x|}\varphi_\epsilon'(x),$$ we have
\begin{align*}
&(1-\frac{m}{2})\int_{B_r}(|\nabla u|^2-\beta(u)|\nabla v|^2)\varphi_\epsilon(x)dx
\\
&=
\int_{B_r}\left(-|\frac{\partial u}{\partial r}|^2+\beta(u)|\frac{\partial v}{\partial r}|^2+\frac{1}{2}(|\nabla u|^2-\beta(u)|\nabla v|^2)\right)|x|\varphi_\epsilon'(x)dx.
\end{align*}
Letting $\epsilon\to 0$, we get
\begin{align*}
&(2-m)\int_{B_r}(|\nabla u|^2-\beta(u)|\nabla v|^2)dx+r\int_{\partial B_r}(|\nabla u|^2-\beta(u)|\nabla v|^2)\\
&=
2r\int_{\partial B_r}(|\frac{\partial u}{\partial r}|^2-\beta(u)|\frac{\partial v}{\partial r}|^2),
\end{align*}
which yields
\begin{align*}
\frac{d}{dr}\left(r^{2-m}\int_{B_r}(|\nabla u|^2-\beta(u)|\nabla v|^2)dx\right)=r^{2-m}\int_{\partial B_r}(|\frac{\partial u}{\partial r}|^2-\beta(u)|\frac{\partial v}{\partial r}|^2).
\end{align*}
The conclusion of the lemma follows by integrating $r$ from $r_1$ to $r_2$.
\end{proof}

\

As a direct corollary of above monotonicity formula, we have
\begin{cor}\label{cor:monotonicity}
Let $(u,v)\in W^{1,2}(\Omega,N\times \R)$ be a stationary Lorentzian harmonic map with Dirichlet boundary data $(\phi,\psi)$. Then for any $x_0\in\Omega$ and $0<r_1\leq r_2<dist(x_0,\partial\Omega)$, there holds
\begin{align*}
r_1^{2-m}\int_{B_{r_1}(x_0)}|\nabla u|^2dx\leq r_2^{2-m}\int_{B_{r_2}(x_0)}|\nabla u|^2dx+C(m,p,\lambda_1,\lambda_2,\Omega,\|\psi\|_{C^0(\partial\Omega)})(r_2)^{2-\frac{2m}{p}}.
\end{align*}
\end{cor}
\begin{proof}
By Lemma \ref{lem:monotonicity}, we have
\begin{align*}
r_1^{2-m}\int_{B_{r_1}(z)}|\nabla u|^2dx&\leq r_2^{2-m}\int_{B_{r_2}(x_0)}(|\nabla u|^2-\beta(u)|\nabla v|^2)dx+r_1^{2-m}\int_{B_{r_1}(x_0)}\beta(u)|\nabla v|^2dx\\&\quad+2\int_{B_{r_2}(x_0)}|x-x_0|^{2-m}\beta(u)|\frac{\partial v}{\partial |x-x_0|}|^2dx\\
&\leq
r_2^{2-m}\int_{B_{r_2}(x_0)}|\nabla u|^2dx+C(m,\lambda_2)(r_2)^{2-\frac{2m}{p}}\|\nabla v\|_{L^p}^2\\
&\leq
r_2^{2-m}\int_{B_{r_2}(x_0)}|\nabla u|^2dx+C(m,p,\lambda_1,\lambda_2,\Omega,\|\psi\|_{C^0(\partial\Omega)})(r_2)^{2-\frac{2m}{p}},
\end{align*}
where the second inequality follows from Young's inequality that
\begin{align}
\int_{B_r}|x|^{2-m}|\nabla v|^2dx\leq \|\nabla v\|_{L^{p}}^2\||x|^{2-m}\|_{L^{\frac{p}{p-2}}(B_r)}\leq C(m,p,\lambda_1,\lambda_2,\Omega,\|\psi\|_{C^0(\partial\Omega)})(r)^{2-\frac{2m}{p}}.
\end{align}
\end{proof}

\

In the  end of this section, we want to recall a regularity theorem for a system of critical PDE in \cite{B-Sharp}. Systems of this form were introduced and studied by \cite{RS}. For this, let us first recall the definition of Morrey spaces (see \cite{Giaquinta}).

\begin{defn}
For $p\geq 1$, $0<\mu \leq m$ and a domain $U\subset \R^m$, the Morrey space $M^{p,\mu}(U)$ is defined by $$M^{p,\mu}(U):=\{f\in L^p_{loc}(U)|\ \|f\|_{M^{p,\mu}(U)}<\infty\}$$ where $$\|f\|^p_{M^{p,\mu}(U)}:=\sup_{B_r\subset U}r^{\mu-m}\int_{B_r}|f|^p.$$
\end{defn}

\

\begin{thm}[Theorem 1.2, \cite{B-Sharp}]\label{thm:01}
For every $m\geq 2$ and $p\in(\frac{m}{2},m)$, there exists $\epsilon=\epsilon(m,d,p)>0$ and $C=C(m,d,p)>0$ with the following property. Suppose that $u\in W^{1,2}(B_1,\R^d)$, $\nabla u\in M^{2,2}(B_1,\R^d)$, $\Omega\in M^{2,2}(B_1,so(d)\otimes\wedge^1\R^m)$ and $f\in L^p(B_1,\R^d)$, satisfy
\begin{equation}\label{equation:04}
\Delta u=\Omega\cdot\nabla u+ f\ in\  B_1
\end{equation} weakly. If $\|\Omega\|_{M^{2,2}(B_1)}\leq \epsilon$, then $$\|\nabla^2u\|_{M^{\frac{2p}{m},2}(B_{1/2})}+\|\nabla u\|_{M^{\frac{2p}{m-p},2}(B_{1/2})}\leq C(\|u\|_{L^1(B_1)}+\|f\|_{L^p(B_1)}).$$
\end{thm}

\

\section{Proof of Theorem \ref{thm:small-energy-regularity} and Theorem \ref{thm:main-1}}\label{sec:proof of mainthm}
In this section, we will prove Theorem \ref{thm:small-energy-regularity} and Theorem \ref{thm:main-1}.

\begin{proof}[\textbf{Proof of Theorem \ref{thm:small-energy-regularity}}]
Without loss of generality, we may assume $r_0=1$ and $$\frac{1}{|B_1|}\int_{B_1}v dx=0.$$ Taking a cut-off function $\eta\in C^\infty_0(B_1)$ such that $0\leq\eta\leq 1$, $\eta|_{B_{7/8}}\equiv 1$ and $|\nabla\eta|\leq C$. By a direct computation, we get $$div(\beta(u)\nabla(\eta v))=div(\beta(u)\nabla\eta v)+\beta(u)\nabla\eta\nabla v\ in \ B_1.$$ Then according to the standard theory of second elliptic operator of divergence form (cf. Theorem 1 in \cite{Meyers}), we have $v\in W^{1,\frac{2m}{m-2}}(B_{7/8})$ and
\begin{align*}
\|\nabla v\|_{L^{\frac{2m}{m-2}}(B_{7/8})}&\leq C(m,\lambda_1,\lambda_2)(\|\nabla\eta v\|_{L^{\frac{2m}{m-2}}(B_1)}+\|\beta(u)\nabla\eta\nabla v\|_{L^2(B_1)})\\
&\leq C(m,\lambda_1,\lambda_2)\|\nabla v\|_{L^2(B_1)},
\end{align*}
where the last inequality follows from Sobolev's embedding $W^{1,2}\hookrightarrow L^{\frac{2m}{m-2}}$ and Poincare's inequality $$\|v\|_{L^2(B_1)}\leq C(m)\|\nabla v\|_{L^2(B_1)}.$$

Using Theorem 1 in \cite{Meyers} and by a bootstrap argument, it is easy to see that $v\in W^{1,p}(B_{3/4})$ for any $1<p<\infty$ and
\begin{align}\label{inequality:01}
\|\nabla v\|_{L^p(B_{3/4})}\leq C(m,p,\lambda_1,\lambda_2)\|\nabla v\|_{L^2(B_1)}.
\end{align}

It is well known that the equation of $u$ can be written as the form of \eqref{equation:04} with $$|\Omega|\leq C(N)|\nabla u|\ \ and\ \ |f|\leq C(\lambda_2,N)|\nabla v|^2.$$ By Theorem \ref{thm:01} and \eqref{inequality:01}, taking $\epsilon_0=\epsilon_0(m,p,N)$ sufficient small, we know $u\in W^{1,p}(B_{5/8})$ for any $m<p<\infty$ and
\begin{align*}
\|\nabla u\|_{L^p(B_{5/8})}&\leq C(m,p,\lambda_1,\lambda_2,N)(\|\nabla u\|_{L^2(B_1)}+\||\nabla v|^2\|_{L^{\frac{mp}{2+p}}(B_1)})\\
&\leq C(m,p,\lambda_1,\lambda_2,N)(\|\nabla u\|_{L^2(B_1)}+\|\nabla v\|^2_{L^2(B_1)}).
\end{align*}
Applying $W^{2,p}$ estimates of Laplacian operator, we obtain
\begin{align*}
  \|\nabla u\|_{W^{1,p}(B_{9/16})}&\leq C(m,p,\lambda_2,N)(\|\nabla u\|^2_{L^{2p}(B_{5/8})}+\|\nabla v\|^2_{L^{2p}(B_{5/8})}+\|\nabla u\|_{L^{2}(B_{5/8})})\\
  &\leq C(m,p,\lambda_1,\lambda_2,N)(\|\nabla u\|_{L^{2}(B_{5/8})}+\|\nabla u\|^2_{L^2(B_1)}+\|\nabla v\|^2_{L^2(B_1)}+\|\nabla v\|^4_{L^2(B_1)})
\end{align*}
and
\begin{align*}
  \|\nabla v\|_{W^{1,p}(B_{9/16})}&\leq C(m,p,\lambda_1,\lambda_2,N)(\||\nabla u||\nabla v|\|_{L^{p}(B_{5/8})}+\|\nabla v\|_{L^{2}(B_{5/8})})\\
 &\leq C(m,p,\lambda_1,\lambda_2,N)\|\nabla v\|_{L^2(B_1)}(1+\|\nabla u\|_{L^2(B_1)}+\|\nabla v\|^2_{L^2(B_1)}).
\end{align*}
By Sobolev's embedding theorem, we see that $(\nabla u,\nabla v)\in C^{\alpha}(B_{9/16})$ for any $\alpha=1-\frac{m}{p}\in (0,1)$ and the estimate \eqref{inequality:09} holds. Then the high regularity follows from the classical Schauder estimates of Laplacian operator and a standard bootstrap argument.
\end{proof}

\

Now, we prove our main Theorem \ref{thm:main-1}.
\begin{proof}[\bf{Proof of Theorem \ref{thm:main-1}}]

Define
\begin{align}
S(u):=
\{x\in\Omega:\liminf_{r\searrow 0}r^{2-n}\int_{B_r(x)}|\nabla u|^2\geq\frac{\epsilon_0^2}{2^m}\}
\end{align}
where $\epsilon_0>0$ is the constant in Theorem \ref{thm:small-energy-regularity}. It is well known that $H^{n-2}(S(u))=0$. Next, we will show $S(u)$ is a closed set and $(u,v)\in C^\infty(\Omega\setminus S(\phi))$.

For any $x_0\in \Omega\setminus S(u)$ and $\epsilon>0$, there exists $0<r_0<\epsilon$ such that,
\begin{align}
(2r_0)^{2-m}\int_{B_{2r_0}(x_0)}|\nabla u|^2dx <\frac{\epsilon_0^2}{2^m}.
\end{align}
Therefore,
\begin{align}
\sup_{z\in B_{r_0}(x_0)}r_0^{2-m}\int_{B_{r_0}(z)}|\nabla u|^2dx\leq r_0^{2-m}\int_{B_{2r_0}(x_0)}|\nabla u|^2dx <\frac{2^{m-2}\epsilon_0^2}{2^m}.
\end{align}

By Corollary \ref{cor:monotonicity}, we have
\begin{align}\label{inequality:08}
\sup_{z\in B_{r_0}(x_0),0<r\leq r_0}r^{2-m}\int_{B_{r}(z)}|\nabla u|^2dx&\leq
\sup_{z\in B_{r_0}(x_0)}r_0^{2-m}\int_{B_{r_0}(z)}|\nabla u|^2dx+C(m,p,\lambda_1,\lambda_2,\|\psi\|_{C^0(\partial\Omega)})(r_0)^{2-\frac{2m}{p}}\notag\\
&\leq
\frac{2^{m-2}\epsilon_0^2}{2^m}+C_1(m,p,\lambda_1,\lambda_2,\|\psi\|_{C^0(\partial\Omega)})(r_0)^{2-\frac{2m}{p}}
\end{align}
for some $m<p<\infty$, where $C_1(m,p,\lambda_1,\lambda_2,\|\psi\|_{C^0(\partial\Omega)})$ is a positive constant.

Taking $\epsilon\leq (\frac{\epsilon_0^2}{4C_1(m,p,\lambda_1,\lambda_2,\|\psi\|_{C^0(\partial\Omega)})})^{\frac{2m}{p}-2}$, we get
\begin{align}\label{inequality:02}
\sup_{z\in B_{r_0}(x_0),0<r\leq r_0}r^{2-m}\int_{B_{r}(z)}|\nabla u|^2dx
\leq \frac{\epsilon_0^2}{2}.
\end{align}
Then Theorem \ref{thm:small-energy-regularity} tells us that $(u,v)\in C^\infty(B_{r_0/2}(x_0))$ which implies $B_{r_0/4}(x_0)\subset \Omega\setminus S(u)$. We finished the proof.
\end{proof}

\

\section{Proof of Theorem \ref{thm:main-2}}\label{sec:gradient estimate}

In this section, we will study the blow-up behavior of a sequence of stationary Lorentzian harmonic map $\{(u_n,v_n)\}$ with Dirichlet boundary data $(\phi,\psi)$ and with bounded energy $$E(u_n,v_n)=\frac{1}{2}\int_\Omega(|\nabla u_n|^2+|\nabla v_n|^2)dx\leq \Lambda.$$ Due to the weak compactness, we may assume $u_n\rightharpoonup u$ weakly in $W^{1,2}(\Omega,N)$ and $$\mu_n:=|\nabla u_n|^2dx\to\mu:=|\nabla u|^2dx+\nu$$ in the sense of Radon measures, where $\nu$ is a nonnegative Radon measure by Fatou's lemma which is usually called the defect measure.

Without loss of generality, we assume $B_1(0)\subseteq\Omega$. Similar to harmonic maps \cite{Lin}, we define the energy concentration set $\Sigma$ as follows
\begin{align}\label{Def:01}
\Sigma=\{x\in B_1(0)|\liminf_{r\searrow 0}\liminf_{n\to\infty}r^{2-n}\int_{B_r(x)}|\nabla u_n|^2dx\geq\frac{\epsilon_0^2}{2^m}\},
\end{align}
where $\epsilon_0$ is the constant in Theorem \ref{thm:small-energy-regularity}.

\

Denoting $spt(\nu)$ the support set of $\nu$ and $$sing(u):=\{x\in B_1(0)|u \mbox{ is not smooth at x }\},$$ then we have
\begin{lem}
Suppose $\{(u_n,v_n)\}$ is a sequence of stationary Lorentzian harmonic map with Dirichlet boundary data $(u_n,v_n)|_{\Omega}=(\phi,\psi)$ and bounded energy $E(u_n,v_n)\leq\Lambda$, then the energy concentration set $\Sigma$ is closed in $B_1$ and $H^{m-2}(\Sigma)\leq C(m,\epsilon_0,\Lambda)$. Moreover, there holds \begin{equation}\label{equation:05}
\Sigma=spt(\nu)\cup sing(u).
\end{equation}
\end{lem}
\begin{proof}
For $x_0\in B_1\setminus\Sigma$, by the definition of $\Sigma$, we know that for any positive constant $$\epsilon\leq (\frac{\epsilon_0^2}{4C_1(m,p,\lambda_1,\lambda_2,\|\psi\|_{C^0(\partial\Omega)})})^{\frac{2m}{p}-2},$$ where $C_1(m,p,\lambda_1,\lambda_2,\|\psi\|_{C^0(\partial\Omega)})$ is the constant in \eqref{inequality:08}, there exists a positive constant $r_0<\epsilon$ and a subsequence of $\{n\}$ (also denoted by $\{n\}$), such that, for any $n$, there holds
\begin{align*}
(2r_0)^{2-m}\int_{B_{2r_0}(x)}|\nabla u_n|^2dx<\frac{\epsilon_0^2}{2^m},
\end{align*}
which implies (similar to deriving \eqref{inequality:02})
\begin{align*}
\sup_{z\in B_{r_0}(x),0<r\leq r_0}r^{2-m}\int_{B_{r}(z)}|\nabla u_n|^2dx<\frac{\epsilon_0^2}{2}.
\end{align*}
By Theorem \ref{thm:main-1}, we know
\begin{align}\label{inequality:03}
\|\nabla u_n\|_{L^\infty(B_{r_0/2}(x_0))}+\|\nabla v_n\|_{L^\infty(B_{r_0/2}(x_0))}\leq C(m,\lambda_1,\lambda_2,\Lambda,N)r_0^{-\frac{m}{2}}.
\end{align}

Then, it is easy to see that there exists a small positive constant $r_1=r_1(m,r_0,\lambda_1,\lambda_2,\Lambda,\epsilon_0,N)$, such that, whenever $r\leq r_1$,
\begin{align*}
\sup_{x\in B_{r_0/4}(x_0)}r^{2-m}\int_{B_{r}(x)}|\nabla u_n|^2dx<\frac{\epsilon_0^2}{2^{m+1}}.
\end{align*}
Thus, $B_{r_0/4}(x_0)\subset B_1\setminus\Sigma$. So, $\Sigma$ is a closed set.

It is standard to get $H^{m-2}(\Sigma)\leq C$ by a covering lemma (cf. \cite{Lin}).

For \eqref{equation:05}, on the one hand, let $x_0\in B_1\setminus \Sigma$. Then \eqref{inequality:03} holds and by standard elliptic estimates of Laplace operator, we have
\begin{align}
\|u_n\|_{C^{1+\alpha}(B_{r_0/4}(x_0))}+\|v_n\|_{C^{1+\alpha}(B_{r_0/4}(x_0))}\leq C,
\end{align}
for some $0<\alpha<1$. Thus, up to a subsequence of $\{u_n,v_n\}$, $u_n\to u$ strongly in $W^{1,2}$ and $u\in C^\infty (B_{r_0/8}(x_0))$ which implies that $x_0\notin sing(u)$ and $x_0\notin spt \nu$ since $\nu\equiv0$ on $B_{r_0/8}(x_0)$.

On the other hand, if $x_0\in\Sigma$, by the definition, for any $r>0$ sufficient small, we have $$\liminf_{n\to\infty}\frac{\mu_n(B_r(x_0))}{r^{m-2}}\geq \frac{\epsilon_0^2}{2^{m+1}},$$ which implies, $$\frac{\mu(B_r(x_0))}{r^{m-2}}\geq \frac{\epsilon_0^2}{2^{m+1}}$$ for $a.e.$ $r>0$. Suppose $x_0\notin sing(\phi)$, then $$r^{2-m}\int_{B_r(x_0)}|\nabla u|^2dx\leq\frac{\epsilon_0^2}{2^{m+2}}$$ whenever $r>0$ is small enough. Then we have $$\frac{\nu(B_r(x_0))}{r^{m-2}}\geq \frac{\epsilon_0^2}{2^{m+2}}$$ for all small positive $r>0$ and $x_0\in spt\nu$. This finishes the proof of lemma.
\end{proof}

\

\begin{lem}
Under the same assumption of above lemma, the limit
\begin{equation}
\theta_\nu (x):=\lim_{r\to  0}\frac{\nu (B_r(x))}{r^{m-2}}
\end{equation}
exists for $H^{m-2}$ a.e. $x\in \Sigma$. Moreover, $$\frac{\epsilon_0^2}{2^m}\leq \theta_\nu (x)\leq C(m,\lambda_1,\lambda_2,\Lambda,N,\|\psi\|_{C^0(\partial \Omega)})\delta_0^{2-m},$$ where $\delta_0:=dist(B_1(0),\partial \Omega)$.
\end{lem}
\begin{proof}
Let $x\in\Omega$ and $s_i\to 0$, $t_i\to 0$ be arbitrary two positive sequence, by Corollary \ref{cor:monotonicity}, we have
\begin{align}\label{inequality:04}
\frac{\mu_n(B_{s_i}(x))}{s_i^{m-2}}\leq \frac{\mu_n(B_{t_j}(x))}{t_j^{m-2}}+C(m,p,\lambda_1,\lambda_2,\Lambda,N,\|\psi\|_{C^0(\partial \Omega)})(t_j)^{2-\frac{2m}{p}}
\end{align}
for $s_i\leq t_j$ and some $m<p<\infty$. Letting firstly $i\to\infty$ and secondly $j\to\infty$, we get
\begin{align*}
\limsup_{r\to 0}\frac{\mu(B_{r}(x))}{r^{m-2}}\leq \liminf_{r\to 0}\frac{\mu(B_{r}(x))}{r^{m-2}}.
\end{align*}
Thus,
\begin{align*}
\lim_{r\to 0}\frac{\mu(B_{r}(x))}{r^{m-2}}
\end{align*}
exists. Noting that for $H^{m-2}$ a.e. $x\in\Omega$,
\begin{align}\label{equation:06}
\lim_{r\to 0}r^{2-m}\int_{B_r(x)}|\nabla u|^2dx=0,
\end{align}
therefore, we have
\begin{align*}
\lim_{r\to 0}\frac{\nu(B_{r}(x))}{r^{m-2}}= \lim_{r\to 0}\frac{\mu(B_{r}(x))}{r^{m-2}}.
\end{align*}

It is easy to see from \eqref{inequality:04} (taking $p=2m$) that \begin{align*}
r^{2-m}\mu(B_r(x))&\leq C(\Lambda)\delta_0^{2-m}+C(m,\lambda_1,\lambda_2,\Lambda,N,\|\psi\|_{C^0(\partial \Omega)})\delta_0\\&\leq C(m,\lambda_1,\lambda_2,\Lambda,N,\|\psi\|_{C^0(\partial \Omega)})\delta_0^{2-m}, \end{align*}
which implies $\mu\lfloor\Sigma$ is absolutely continuous with respect to $H^{m-2}\lfloor\Sigma$. By Radon-Nikodym theorem, we know that there exists a measurable function $\theta(x)$ such that $$\mu\lfloor\Sigma=\theta(x)H^{m-2}\lfloor\Sigma.$$

Noting that for $H^{m-2}$ a.e. $x\in\Sigma$, $$2^{2-m}\leq\liminf_{r\to 0}\frac{H^{m-2}(\Sigma\cap B_r(x))}{r^{m-2}}\leq\limsup_{r\to 0}\frac{H^{m-2}(\Sigma\cap B_r(x))}{r^{m-2}}\leq 1$$ and \eqref{equation:06}, we have $$\nu\lfloor\Sigma=\theta(x)H^{m-2}\lfloor\Sigma$$ and $$\frac{\epsilon_0^2}{2^m}\leq \theta_\nu (x)=\theta(x)\leq C(m,\lambda_1,\lambda_2,\Lambda,N,\|\psi\|_{C^0(\partial \Omega)})\delta_0^{2-m}.$$
\end{proof}

\

Since $\nu$ is absolutely continuous with respect to $H^{m-2}\lfloor\Sigma$ and $\nu=0$ outside $\Sigma$, $\theta_\nu(x)$ is positive for $\nu$-a.e. $x\in\Omega$. Hence by Preiss's results \cite{Preiss}, we have
\begin{cor}
The set of energy concentration points $\Sigma$ is $(m-2)$-rectifiable.
\end{cor}

\

For any $y\in\Sigma$ and $\lambda>0$, we define a scaled Radon measure $\mu_{y,\lambda}$ by $$\mu_{y,\lambda}(A)=\lambda^{2-m}\mu(y+\lambda A).$$ A Radon measure $\mu_*$ is called the tangent measure of $\mu$ at $y$ if $$\mu_{y,\lambda}\to\mu_*$$ in the sense of Radon measure as $r\searrow 0$ (See \cite{Federer,Simon}.).

\

\begin{lem}\label{lem:-03}
Suppose $H^{m-2}(\Sigma)>0$, then there exists a nonconstant harmonic sphere $S^2$ into $N$.
\end{lem}
\begin{proof}
Since $\Sigma$ is $(m-2)$-rectifiable and $ H^{m-2}(\Sigma)>0$, we know there exists a point $x_0\in\Sigma$, such that, $\nu$ has a tangent measure $\nu_*$ at $x_0$ and $$\nu_*=\theta_\nu(x_0)H^{m-2}\lfloor\Sigma_*$$ where $\Sigma_*\subset\R^m$ is a $(m-2)$ linear subspace which is usually called the tangent space of $\Sigma$ at $x_0$. Without loss of generality, we may assume $x_0=0$ and $\Sigma_*=\R^{m-2}\times\{(0,0)\}$.

By a similar diagonal argument as that in \cite{Lin}, there exists a sequence $r_n\to 0$, such that, $$\widetilde{\mu}_n^1:=|\nabla \widetilde{u}_n^1|^2dx\to \nu_*$$ in the sense of Radon measure, where $\widetilde{u}_n^1(x):=u_n(x_0+r_nx).$

Set $\widetilde{v}_n^1(x):=v_n(x_0+r_nx)$, it is easy to see that $(\widetilde{u}^1_n,\widetilde{v}^1_n)$ is also a stationary Lorentzian harmonic map. By Lemma \ref{lem:monotonicity}, we have
\begin{align}\label{equation:09}
&r_2^{2-m}\int_{B_{r_2}(0)}(|\nabla \widetilde{u}^1_n|^2-\beta(\widetilde{u}^1_n)|\nabla \widetilde{v}^1_n|^2)dx-r_1^{2-m}\int_{B_{r_1}(0)}(|\nabla \widetilde{u}^1_n|^2-\beta(\widetilde{u}^1_n)|\nabla \widetilde{v}^1_n|^2)dx\notag\\
&=
2\int_{r_1}^{r_2}r^{2-m}\int_{\partial B_r(0)}(|\frac{\partial \widetilde{u}^1_n}{\partial |x|}|^2-\beta(\widetilde{u}^1_n)|\frac{\partial \widetilde{v}^1_n}{\partial |x|}|^2)dH^{n-1}dr.
\end{align}

By Young's inequality, there holds
\begin{align}
\int_{B_r}|x|^{2-m}|\nabla \widetilde{v}^1_n|^2dx&\leq (r_n)^{2-\frac{2m}{p}}\|\nabla v_n\|_{L^{p}(B_{r_nr})}^2\||x|^{2-m}\|_{L^{\frac{p}{p-2}}(B_r)}\notag\\&\leq C(m,p,\lambda_1,\lambda_2,\Lambda,N,\|\psi\|_{C^0(\partial \Omega)})(r_nr)^{2-\frac{2m}{p}}.
\end{align}
Letting $n\to\infty$ in \eqref{equation:09} and noting that $$r_2^{2-m}\nu_*(B_{r_2}(0))=r_1^{2-m}\nu_*(B_{r_1}(0)),$$ we get
\begin{equation}
\lim_{n\to\infty}\int_{B_2(0)}|\frac{\partial \widetilde{u}_n^1}{\partial |x|}|^2dx=0.
\end{equation}

Similarly, since $\nu_{*y,r}=\nu_*$ for any $y\in\Sigma_*$ and $r>0$, we also have
\begin{equation}
\lim_{n\to\infty}\int_{B_2(0)}|\frac{\partial \widetilde{u}_n^1}{\partial |x-y|}|^2dx=0,\ for\ y\in \Sigma_*\cap B_2.
\end{equation}
These implies,
\begin{equation}\label{equation:10}
\lim_{n\to\infty}\sum_{k=1}^{m-2}\int_{B_2(0)}|\frac{\partial \widetilde{u}_n^1}{\partial x^k}|^2dx=0.
\end{equation}

Let $x'=(x_1,...,x_{m-2}),x''=(x_{m-1},x_m)$, define $f_n:B^{m-2}_{1}\to\R$ by $$f_n(x'):=\sum_{k=1}^{m-2}\int_{B^2_{1}(0)}|\frac{\partial \widetilde{u}_n^1}{\partial x_k}|^2(x',x'')dx''.$$ Then, \eqref{equation:10} tells us $$\lim_{n\to\infty}\|f_n(x')\|_{L^1(B^{m-2}_{1}(0))}=0.$$ Denote $M(f_n)(x')$ as the Hardy-Littlewood maximal function, i.e. $$M(f_n)(x)=\sup_{0<r<\frac{1}{2}}r^{2-m}\int_{B^{m-2}_r(x)}f_n(x')dx',\ x\in B^{m-2}_{1/2}(0).$$ By the weak $L^1-$estimate, for any $\rho>0$, we have
\begin{align*}
|\{x\in B^{m-2}_{1/2}(0)|M(f_n)>\rho\}|\leq\frac{C(m)}{\rho}\|f_n\|_{L^1(B^{m-2}_{1/2}(0))},
\end{align*}
which implies
\begin{align*}
|\{x\in B^{m-2}_{1/2}(0)|\limsup_{n\to\infty}M(f_n)>0\}|=0.
\end{align*}

Combing this with Theorem \ref{thm:main-1}, we know there exists a sequence of points $\{x'_n\in B^{m-2}_{1/2}(0)\}$, such that $(\widetilde{u}^1_n,\widetilde{v}_n^1)$ is smooth near $(x_n',x'')$ for all $x''\in B^2_{1}(0)$ and
\begin{equation}\label{equation:07}
\lim_{n\to\infty}M(f_n)(x'_n)=0.\end{equation}

By the blow-up argument in \cite{Lin}, we can find sequences $\{\sigma_n\}$ and $\{x''_n\}\subset B^2_{1/2}(0)$ such that $\sigma_n\to 0$, $x''_n\to (0,0)$ and
\begin{equation}\label{equation:08}
\max_{x''\in B^2_{1/2}(0)}\sigma_n^{2-m}\int_{B^{m-2}_{\sigma_n}(x'_n)\times B^2_{\sigma_n}(x'')}|\nabla \widetilde{u}_n^1|^2dx=\frac{\epsilon_0^2}{C_2(m)},
\end{equation}
where the maximum is achieved at the point $x''_n$ and $C_2(m)>2^{m}$ is a positive constant to be determined later.

In fact, denote $$g_n(\sigma):=\max_{x''\in B^2_{1/2}(0)}\sigma^{2-m}\int_{B^{m-2}_{\sigma}(x'_n)\times B^2_{\sigma}(x'')}|\nabla \widetilde{u}_n^1|^2dx.$$ On the one hand, noting that $(u_n,v_n)$ is smooth near $x'_n\times B^2_{1}(0)$, then we have $$\lim_{\sigma\to 0}g_n(\sigma)=0.$$ On the other hand, for any $\sigma>0$, when $n$ is big enough, there must hold that $g_n(\sigma)\geq\frac{\epsilon_0^2}{2^m}$. For otherwise, by Theorem \ref{thm:small-energy-regularity}, $\widetilde{u}^1_n$ will converge strongly in $W^{1,2}$ to a constant map, which is contradict to $\widetilde{\mu}_n\to \nu_*$. Thus, there exists $\sigma_n$, such that $g_n(\sigma_n)=\frac{\epsilon_0^2}{C_2(m)}$ and we may assume the maximum is achieved at $x''_n$. Next, we show $\sigma_n\to 0$ and $x''_n\to (0,0)$.

If $\sigma_n\geq\delta>0$, by Corollary \ref{cor:monotonicity}, we have
\begin{align*}
\frac{\epsilon_0^2}{C_2(m)}=\limsup_{n\to\infty}g_n(\sigma_n)\geq\limsup_{n\to\infty}\left( g_n(\delta)-C(m,p,\lambda_1,\lambda_2,\Omega,\|\psi\|_{C^0(\partial\Omega)})(r_n\delta)^{2-\frac{2m}{p}}\right)\geq\frac{\epsilon_0^2}{2^m},
\end{align*}
which is a contradiction.

If $x''_n\to x''_0\in B^2_{1/2}(0)$ and $x_0''\neq (0,0)$, for any $\sigma<\frac{|x_0''|}{2}$,
\begin{align*}
\frac{\epsilon_0^2}{2^m}\leq\limsup_{n\to\infty}g_n(\sigma)\leq\sigma^{2-m}\nu_*(B^{m-2}_{1}(0)\times B^2_{2\sigma}(x_0''))=0.
\end{align*}
This is also a contradiction.

Let $x_n=(x'_n,x''_n)$ and $$(\widetilde{u}_n^2(x),\widetilde{v}^2_n(x)):=(\widetilde{u}_n^1(x_n+\sigma_nx),\widetilde{v}_n^1(x_n+\sigma_nx)).$$
Then $(\widetilde{u}_n^2(x),\widetilde{v}_n^2(x))$ is a stationary Lorentzian harmonic map defined on $B^{m-2}_{R_n}(0)\times B^2_{R_n}(0)$, where $R_n=\frac{1}{4\sigma_n}$ which tends to infinite as $n\to\infty$.

By \eqref{equation:07}, we have
\begin{align}\label{inequality:05}
&\lim_{n\to\infty}\sup_{0<R<R_n}R^{2-m}\int_{B^{m-2}_{R}(0)\times B^2_{R_n}(0)}\sum_{k=1}^{m-2}|\frac{\partial \widetilde{u}^2_n}{\partial x_k }|^2dx\notag\\
&=\lim_{n\to\infty}\sup_{0<R<R_n}(\sigma_nR)^{2-m}\int_{B^{m-2}_{\sigma_nR}(x'_n)\times B^2_{\sigma_n R_n}(x_n'')}\sum_{k=1}^{m-2}|\frac{\partial \widetilde{u}_n^1}{\partial x_k }|^2dx\notag\\
&\leq\lim_{n\to\infty}M(f_n)(x_n')=0.
\end{align}

By \eqref{equation:08}, we get
\begin{align}\label{inequality:06}
\frac{\epsilon_0^2}{C_2(m)}=\int_{B^{m-2}_{1}(0)\times B^2_{1}(0)}|\nabla \widetilde{u}^2_n|^2dx=\max_{x''\in B^2_{R_n-1}(0)}\int_{B^{m-2}_{1}(0)\times B^2_{1}(x'')}|\nabla \widetilde{u}^2_n|^2dx.
\end{align}

By Corollary \ref{cor:monotonicity}, for any $R>0$, we obtain
\begin{align}\label{inequality:07}
\int_{B^{m-2}_{R}(0)\times B^2_{R}(0)}|\nabla \widetilde{u}^2_n|^2dx&=(\sigma_n)^{2-m}\int_{B^{m-2}_{\sigma_nR}(x'_n)\times B^2_{\sigma_n R}(x_n'')}|\nabla \widetilde{u}_n^1 |^2dx\notag\\
&\leq C(m,\lambda_1,\lambda_2,\delta_0,\Lambda,\Omega,\|\psi\|_{C^0(\partial \Omega)})R^{m-2},
\end{align}
when $n$ is big enough.

Let $\zeta\in C^\infty_0(B^{m-2}_1(0))$ and $\eta\in C^\infty_0(B^{2}_1(0))$ be two cut-off functions such that $0\leq\zeta\leq 1$, $\zeta|_{B^{m-2}_{1/2}(0)}\equiv 1$, $0\leq\zeta\leq 1$ and $\eta|_{B^{2}_{1/2}(0)}\equiv 1$. Similar to \cite{Lin}, for any $R>0$, we define $F_n(a):B^{m-2}_6(0)\times B^{2}_{R}(0)\to\R$ as follows: $$F_n(a)=\int_{B^{m-2}_1(0)\times B^{2}_1(0)}|\nabla \widetilde{u}_n^2|^2(a+x)\zeta(x')\eta(x'')dx.$$ Computing directly, one has
\begin{align*}
\frac{\partial F_n(a)}{\partial a_k}&=\int_{B^{m-2}_1(0)\times B^{2}_1(0)}\frac{\partial}{\partial x_k}|\nabla \widetilde{u}_n^2|^2(a+x)\zeta(x')\eta(x'')dx\\
&=2\int_{B^{m-2}_1(0)\times B^{2}_1(0)}\langle \frac{\partial \widetilde{u}^2_n}{\partial x_l}, \frac{\partial^2 \widetilde{u}^2_n}{\partial x_l\partial x_k}\rangle(a+x)\zeta(x')\eta(x'')dx\\
&=-2\int_{B^{m-2}_1(0)\times B^{2}_1(0)}\langle \Delta \widetilde{u}^2_n, \frac{\partial \widetilde{u}^2_n}{\partial x_k}\rangle(a+x)\zeta(x')\eta(x'')dx\\&\quad-2\int_{B^{m-2}_1(0)\times B^{2}_1(0)}\langle \frac{\partial \widetilde{u}^2_n}{\partial x_l}, \frac{\partial \widetilde{u}^2_n}{\partial x_k}\rangle(a+x)\frac{\partial}{\partial x_l}(\zeta(x')\eta(x''))dx.
\end{align*}

On the one hand, by \eqref{equation:LHM}, we have
 \begin{align*}
&-2\int_{B^{m-2}_1(0)\times B^{2}_1(0)}\langle \Delta \widetilde{u}_n^2, \frac{\partial \widetilde{u}_n^2}{\partial x_k}\rangle(a+x)\zeta(x')\eta(x'')dx\\
&=-2\int_{B^{m-2}_1(0)\times B^{2}_1(0)}\langle B^\top(\widetilde{u}_n^2)|\nabla \widetilde{v}_n^2|^2, \frac{\partial \widetilde{u}_n^2}{\partial x_k}\rangle(a+x)\zeta(x')\eta(x'')dx\\
&\leq
C(\int_{B^{m-2}_{R+1}(0)\times B^{2}_{R+1}(0)}|\nabla \widetilde{v}_n^2|^4dx)^{1/2}(\int_{B^{m-2}_{R+1}(0)\times B^{2}_{R+1}(0)}|\frac{\partial \widetilde{u}_n^2}{\partial x_k}|^2dx)^{1/2}.
 \end{align*}

On the other hand, by Holder's inequality, one has
\begin{align*}
&-2\int_{B^{m-2}_1(0)\times B^{2}_1(0)}\langle \frac{\partial \widetilde{u}_n^2}{\partial x_l}, \frac{\partial \widetilde{u}_n^2}{\partial x_k}\rangle(a+x)\frac{\partial}{\partial x_l}(\zeta(x')\eta(x''))dx\\
&\leq
C(\int_{B^{m-2}_{R+1}(0)\times B^{2}_{R+1}(0)}|\nabla \widetilde{u}_n^2|^2dx)^{1/2}(\int_{B^{m-2}_{R+1}(0)\times B^{2}_{R+1}(0)}|\frac{\partial \widetilde{u}_n^2}{\partial x_k}|^2dx)^{1/2}.
\end{align*}
Combing these together and letting $n\to\infty$, we obtain
\begin{align*}
\frac{\partial F_n(a)}{\partial a_k}\to 0,\ k=1,...,m-2,
\end{align*}
uniformly in $B^{m-2}_2(0)\times B^{2}_{R}(0)$ for any fixed $R>0$.

Thus, for any $a=(a',a'')=B^{m-2}_6(0)\times B^{2}_{R}(0)$,
\begin{align*}
\int_{B^{m-2}_{1/2}(a')\times B^{2}_{1/2}(a'')}|\nabla\widetilde{u}_n^2|^2dx&\leq F_n(a)\\
&\leq F_n((0,a''))+C(m)\sum_{k=1}^{m-2}|\frac{\partial F_n(a)}{\partial a_k}|\\
&\leq \int_{B^{m-2}_{1}(0)\times B^{2}_{1}(a'')}|\nabla\widetilde{u}_n^2|^2dx+C(m)\sum_{k=1}^{m-2}|\frac{\partial F_n(a)}{\partial a_k}|\\
&\leq\frac{\epsilon_0^2}{C_2(m)}+C(m)\sum_{k=1}^{m-2}|\frac{\partial F_n(a)}{\partial a_k}|.
\end{align*}

Therefore, when $n$ is big enough, we have
\begin{align}
6^{2-m}\int_{B^{m-2}_6(0)\times B^2_6(0)}|\nabla\widetilde{u}_n^2|^2(x',x''+b)dx\leq \frac{C(m)\epsilon_0^2}{C_2(m)} \ for \ all \ b\in B^2_{R}(0).
\end{align}
Taking $C_2(m)\geq 2^{m}C(m)$, by Corollary \ref{cor:monotonicity}, we have
\begin{align*}
&\sup_{x_0\in B_3(0),0<r\leq 3}r^{2-m}\int_{B_r(x_0)}|\nabla\widetilde{u}_n^2|^2(x',x''+b)dx\\
&\leq
\sup_{x_0\in B_3(0)}3^{2-m}\int_{B_3(x_0)}|\nabla\widetilde{u}_n^2|^2(x',x''+b)dx+C(m,p,\lambda_1,\lambda_2,\Omega,\|\psi\|_{C^0(\partial M)})(\sigma_nr_n)^{2-\frac{2m}{p}}\\
&\leq
2^{m-2}6^{2-m}\int_{B^{m-2}_6(0)\times B^2_6(0)}|\nabla\widetilde{u}_n^2|^2(x',x''+b)dx+C(m,p,\lambda_1,\lambda_2,\Omega,\|\psi\|_{C^0(\partial M)})(\sigma_nr_n)^{2-\frac{2m}{p}}\\&\leq \frac{2^{m-2}C(m)\epsilon_0^2}{C_2(m)}+C(m,p,\lambda_1,\lambda_2,\Omega,\|\psi\|_{C^0(\partial M)})(\sigma_nr_n)^{2-\frac{2m}{p}}\leq \frac{\epsilon_0^2}{2},
\end{align*}
for some $m<p<\infty$, whenever $n$ is large enough.

By Theorem \ref{thm:small-energy-regularity}, we know $(\widetilde{u}_n^2,\widetilde{v}_n^2)$ sub-converges to a Lorentzian harmonic map $(\widetilde{u},\widetilde{v})$ in $C^1_{loc}(B^{m-2}_{3/2}(0)\times \R^2)$. Moreover, by \eqref{inequality:05}-\eqref{inequality:07}, for any $R>0$, we have
\begin{align*}
\int_{B_R(0)}\sum_{k=1}^{m-2}|\frac{\partial \widetilde{u}}{\partial x_k}|^2dx=0,
\end{align*}
and
\begin{align*}
\int_{B_1(0)}|\nabla \widetilde{u}|^2dx=\frac{\epsilon_0^2}{C_2(m)},\quad       \int_{B_R(0)}|\nabla \widetilde{u}|^2dx\leq C(m,\lambda_1,\lambda_2,\delta_0,\Lambda,\Omega,\|\psi\|_{C^0(\partial \Omega)})R^{m-2}.
\end{align*}

Furthermore, since
\begin{align*}
\int_{B_R(0)}|\nabla\widetilde{v}|^2dx=\lim_{n\to\infty}\int_{B_R(0)}|\nabla\widetilde{v}^2_n|^2dx\leq\lim_{n\to\infty} (r_n\sigma_n)^{2-\frac{2m}{p}}R^{m(1-\frac{2}{p})}\|\nabla v_n\|_{L^p}^2=0,
\end{align*}
we know $\widetilde{v}$ is a constant and $\widetilde{u}:\R^2\to N$ is a nonconstant harmonic map with finite energy. By the conformal theory of harmonic map in dimension two, $\widetilde{u}$ can be extended to a nonconstant harmonic sphere.
\end{proof}

\

\begin{proof}[\textbf{Proof of Theorem \ref{thm:main-2}}]
The conclusion of Theorem \ref{thm:main-2} standardly follows from Lemma \ref{lem:-03} and the Federer dimensions reduction argument which is  similar to \cite{Schoen-Uhlenbeck} for minimizing harmonic maps. We omit the details here. This completes the proof.
\end{proof}

\

\vskip 0.2cm


\providecommand{\bysame}{\leavevmode\hbox to3em{\hrulefill}\thinspace}
\providecommand{\MR}{\relax\ifhmode\unskip\space\fi MR }
\providecommand{\MRhref}[2]{%
  \href{http://www.ams.org/mathscinet-getitem?mr=#1}{#2}
}
\providecommand{\href}[2]{#2}

\end{document}